\documentclass[11pt]{amsart}
\usepackage{amsmath,amssymb,amsbsy,amsfonts,amsthm,latexsym,amsopn,amstext,amsxtra,epic, euscript,amscd,indentfirst,verbatim,xcolor,graphicx}
\begin{document}

\newtheorem{theorem}{Theorem}
\newtheorem{conjecture}[theorem]{Conjecture}
\newtheorem{proposition}[theorem]{Proposition}
\newtheorem{question}[theorem]{Question}
\newtheorem{definition}{Definition}
\newtheorem{lemma}[theorem]{Lemma}
\newtheorem{cor}[theorem]{Corollary}
\newtheorem{obs}[theorem]{Observation}
\newtheorem{proc}[theorem]{Procedure}
\newcommand{\comments}[1]{} 
%% DEFINITIONS
\def\Z{\mathbb Z}
\def\Za{\mathbb Z^\ast}
\def\Fq{{\mathbb F}_q}
\def\R{\mathbb R}
\def\N{\mathbb N}
\def\k{\kappa}
\newcommand{\todo}[1]{\textbf{\textcolor{red}{[To Do: #1]}}}

\title[Emergence of flat points]{Curvature-Torsion Entropy for Twisted Curves under Curve Shortening Flow}

\author{Gabriel Khan}
\email{gkhan@iastate.edu}

\date{\today}

\maketitle

\begin{abstract}
We study curve-shortening flow for twisted curves in $\mathbb{R}^3$ (i.e., curves with nowhere vanishing curvature $\kappa$ and torsion $\tau$) and define a notion of torsion-curvature entropy. Using this functional, we show that either the curve develops an inflection point or the eventual singularity is highly irregular (and likely impossible). In particular, it must be a Type II singularity which admits sequences along which $\frac{\tau}{\kappa^2} \to \infty$. This contrasts strongly with Altschuler's planarity theorem \cite{A}, which shows that along any essential blow-up sequence, $\frac{\tau}{\kappa} \to 0$.
\end{abstract}

\section{Introduction}

\emph{Curve shortening flow} is the geometric flow defined by the equation
\begin{equation}\label{eq:CSF}
\partial_t {\gamma} = \kappa N,
\end{equation}
where $\gamma$ is a smooth immersed curve in $\mathbb{R}^n$, $\kappa$ is the curvature and $N$ is the unit normal vector. Solutions to this flow consist of a family of curves $\gamma_t$ with $t \in [0,~\omega)$ with $\gamma_0$ (which we will often denote as $\gamma$) as the initial condition. 

This flow was introduced by Gage and Hamilton in 1986 as the $L^2$ gradient flow for the arc length (i.e., the flow which shortens curves the quickest) \cite{GH}. Their work established short-time existence and uniqueness for the flow and showed that if one starts with a closed convex curve in the plane, the curve shrinks to a point while becoming asymptotically round. Put succinctly, convex curves shrink to round points. The following year, Grayson \cite{G} proved that any curve which is initially embedded (i.e., does not self-intersect) in $\mathbb{R}^2$ eventually becomes convex under the flow, and thus converges to a round point.

In two dimensions, curve shortening flow has two fundamental properties which play a crucial role in the analysis. First, if a curve starts as an embedded curve, it remains so until it reaches a singularity.\footnote{More generally, if one considers mean-curvature flow for co-dimension one hypersurfaces, the flow is self-avoiding. In other words, surfaces that do not intersect initially will never intersect in the future.} Second, the number of inflection points is non-increasing under the flow. Both of these facts can be shown by a straightforward application of the parabolic maximum principle but fail for curve shortening flow in higher dimensions.

\subsection{Singularity Formation}

Since curve shortening flow shrinks the length of curves, a closed curve must encounter a singularity at some time $\omega$.\footnote{For closed and embedded curves in the plane, the area decreases linearly at the rate of $-2 \pi$, so it is possible to compute the time $\omega$ explicitly.} The singularities are fairly well understood in two dimensions, and a natural question is to extend these results to curve shortening flow in dimensions three or higher. To discuss this further, we first introduce some notation.

Curve shortening flow exists so long as the curvature is bounded. To study singularities, we consider a \emph{blow-up sequence}, which is a sequence of points in space-time $(p_m, t_m)$ such that the curvature at $(p_m, t_m)$ goes to infinity. A blow-up sequence is said to be \emph{essential} if $\kappa^2(p_m, t_m) \ge \rho M_{t_m}$ for some $\rho > 0$ where  \begin{equation} M_{t} = \sup_{p \in \gamma} \kappa^2(p, t). \end{equation}

One can divide the singularities of curve shortening flow into two broad classes, \emph{Type I} and \emph{Type II}. A singularity is said to be Type I if \[\limsup_{t \to \omega} M_{t} \cdot (\omega - t)\] is bounded and Type II otherwise. Type I singularities are \emph{global} singularities, in that the entire curve shrinks to a point while converging in $C^\infty$ to a homothetic (i.e., self-similar) shrinking curve. For closed curves, the possible models for these singularities were classified by Abresch and Langer \cite{AL}. On the other hand, Type II singularities are \emph{local} in that the curvature goes to infinity in a small region while possibly remaining bounded elsewhere. Such singularities appear as kinks in the curve and admit an essential blow-up sequence which (after rescaling) converges in $C^\infty$ to the Grim Reaper curve $y = -\log(\cos x)$, which is a translating soliton under the flow.

As we have mentioned, curve shortening flow behaves differently in higher dimensions, which complicates its analysis. However, Altschuler \cite{A} established that under the flow, curves in three-dimensional space becomes asymptotically \emph{planar} near any singularity. More precisely, along any essential blow-up sequence $(p_m,t_m)$, the torsion $\tau$ satisfies 
\[ \lim_{m \to \infty} \frac{\tau}{\k} (p_m,t_m)=0.\] This result was then extended to curve shortening flow in $\mathbb{R}^n$ by Yan and Jiao \cite{YanJiao}. As such, the singularity models for curve shortening flow in higher dimensions are the same as for curve shortening flow in the plane, although we cannot rule out the appearance of Abresch-Langer solutions or Type II singularities even when the initial curve is embedded.

\section{Twisted Curves and the Curvature-Torsion Entropy}

\begin{figure}
    \centering
\includegraphics[width=.7\linewidth]{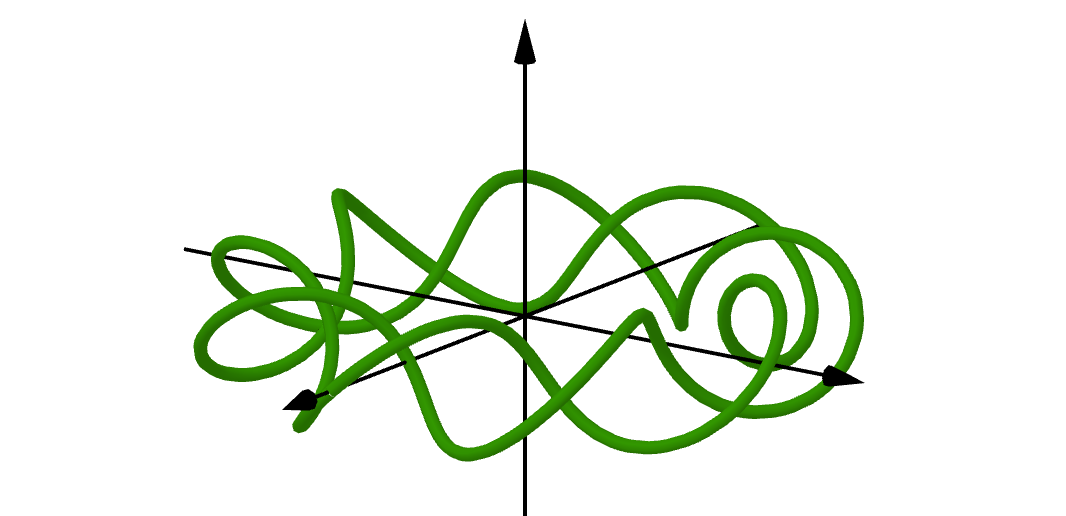}
\caption{A twisted curve}
\label{fig:A twisted curve}
\end{figure}

A curve in $\mathbb{R}^n$ is said to be \emph{twisted} if it has linearly independent derivatives up to order $n$ \cite{Costa}. In three dimensions, this corresponds to the nowhere vanishing of the curvature and torsion along the curve. The main focus of this paper is to study curve shortening flow for twisted curves (i.e., curves which are twisted). In particular, we focus on the possible singularities which emerge when a curve is twisted.

\begin{theorem} \label{Main Theorem}
    Suppose $\gamma_t$ is a solution to curve shortening flow which is twisted up to the time of singularity $\omega$. Then we have the following.
    \begin{enumerate}
        \item The singularity is Type II.
        \item There exists a sequence $(p_m,t_m)$ so that $t_m \to \omega$ and
    \begin{equation}
        \frac{\tau}{\kappa^2}(p_m,t_m) \to \infty.
    \end{equation}
    \end{enumerate}
\end{theorem}

Note that the sequences where $\frac{\tau}{\kappa^2}(p_m,t_m) \to \infty$ cannot be essential, and may not even be a blow-up sequence at all. To establish this fact, we find a quantity that is (nearly) increasing under the flow. In particular, we will consider the following entropy functional, whose behavior will be well-controlled under the flow.
\begin{definition}
For a twisted curve in $\mathbb{R}^3$, the curvature-torsion entropy is defined to be the quantity
\begin{equation}\label{CurvatureTorsionEntropy}
     \int_\gamma \kappa \log\left( \frac{\tau}{\kappa^2} \right) \, ds.
\end{equation}
\end{definition}
One could extend this definition to non-twisted curves by squaring the argument of the logarithm. However, this integral will essentially always be $-\infty$ whenever there is a flat point (i.e., a point with $\tau=0$), so we will not consider this generalization.

%For a twisted curve, we will show that this quantity cannot decrease too quickly when $\frac{\tau}{\kappa^2}$ is bounded from above. However, in either case, the curvature-torsion entropy tends to $-\infty$ near a singularity, which yields a contradiction.

\section{Nearly monotonic functionals}
To prove the main result, we must show that the curvature-torsion entropy is nearly non-decreasing under the flow. Before doing so, we start with a simpler proof that a twisted curve cannot develop a Type I singularity. This result was previously shown in unpublished work of the author \cite{Khan}. Since the proof is very short, we include it here.

\begin{theorem}
\label{Type II singularities}
Suppose $\gamma$ is a twisted curve in $\mathbb{R}^3$. Under curve shortening flow, one of the following two possibilities occurs.
\begin{enumerate}
    \item There exists a time $t_0$ where $\gamma_{t_0}$ has a point with vanishing curvature. Furthermore, after this time $\gamma_t$ has a flat point until the singular time.
    \item $\gamma_t$ develops a Type II singularity.
\end{enumerate}
\end{theorem}

\begin{proof}
Consider any curve $\gamma$ (not necessarily twisted) which develops a Type I singularity at time $\omega$. After rescaling, $\gamma_t$ approaches an Abresch-Langer solution \cite{AL} with finite winding number in the $C^\infty$ sense \cite{A}. This has two consequences.
\begin{enumerate}
    \item All blow up sequences are essential. In other words, for times close to the singularity, the maximum curvature is a bounded multiple of the minimum curvature. As a result, there exists a time $t_0 \in [0, \omega)$ such that for all times afterward, $\gamma_t$ has no inflection points, which implies that torsion is defined everywhere on the curve. 
    \item Since $\gamma$ converges to some Abresch-Langer curve, the functional $D(t) = \sup \k_t \cdot L_t$ converges to a finite limit as the times goes to $\omega$, where $L_t$ is the length of the curve. In particular, this quantity remains bounded.
\end{enumerate}

Now we consider a solution $\gamma_t$ which is twisted after $t_0$. For the following calculations, we parameterize $\gamma_t$ smoothly by $u \in [0, 2\pi)$ and suppose that the curvature has velocity $v$. To establish the result, we show that for a twisted curve, the total torsion is increasing.
\begin{equation}\label{TotalTorsion}
    \int_{\gamma_t} \tau  \, ds = \int_0^{2\pi} \tau \cdot v \, du.
\end{equation}

To compute the evolution of this quantity, we use the evolution equations for $v$, $\kappa$ and $\tau$ (derived in \cite{A} and \cite{AG}).
\begin{eqnarray}
\partial_t v & = & -\kappa^2 v \nonumber \\
\partial_t \kappa &=& \partial_s^2 \k + \k^3 -\k\tau^2  \label{eq:Curvature Evolution}  
 \\
\partial_t \tau&=& 2\k^2\tau + \partial_s \left(\frac{2\tau}{\kappa}
\partial_s \kappa \right) + \partial_s^2 \tau. \label{eq:Torsion Evolution} \nonumber
\end{eqnarray}

Using these, we find the following:
\begin{eqnarray*}
 \partial_t \int_{\gamma_t} \tau\cdot v \, du  & = & \int_0^{2\pi} (\partial_t \tau) \cdot v + (\partial_t v) \cdot \tau \, du\\
& = &  \int_0^{2\pi} \left( 2\k^2\tau + \partial_s \left(\frac{2\tau}{\kappa} 
\partial_s \kappa \right) + \partial_s^2 \tau \right)  v -\k^2 \, v \, \tau \, du\\
%& = &  \int_0^{2\pi} \k^2\, \tau \, v \, du + \int_0^{2\pi}  \left(\partial_s \left(\frac{2\tau}{\kappa} 
%\partial_s \kappa \right) + \partial_s^2 \tau\right)  v \, du \\
 & = & \int_{\gamma_t} \k^2\, \tau \, ds + \int_\gamma \partial_s \left(\frac{2\tau}{\kappa} 
\partial_s \kappa \right) + \partial_s^2 \tau \, ds \\
 & = & \int_{\gamma_t} \k^2\, \tau \, ds
%& = & \int_{\gamma_t} \k^2\, \tau \, ds + \left(\frac{2\tau}{\kappa} 
%\partial_s \kappa  + \partial_s \tau\right)|_{\partial {\gamma_t}}\\ 
\end{eqnarray*}

%Since the boundary of ${\gamma_t}$ is empty,
%$$ \partial_t ||\tau||_1 = \int_{{\gamma_t}_t} \k^2\cdot |\tau| \, ds > 0.$$

Therefore, the $L^1$ norm of $\tau$ is increasing and approaches a positive (possibly infinite) limit as $t$ goes to $\omega$. However, \[\sup_{p \in \gamma_t} \tau (p)\cdot L_t \geq ||\tau||_1(t) > 0. \] 
When combined with the fact that $D(t)$ remains bounded, this implies that
 \[\lim_{t \to \omega}\sup_{p \in \gamma_t} \frac{\tau}{\k} > 0,\]
which contradicts the planarity theorem. As a result, the curve must develop a Type II singularity. To complete the proof, we show that if a curve is twisted, it cannot become untwisted unless a point of zero curvature appears.
\begin{proposition}
Suppose we have a family of curves $\gamma_t$ in $\R^3$ which satisfy Equation \eqref{eq:CSF} and that $\gamma_0$ is twisted. Furthermore, suppose that $\gamma_t$ has no inflection points for $t\in[0,t_0]$. Then for all $t\in[0,t_0]$, $\gamma_t$ is also twisted.
\end{proposition}
    This proposition follows from the maximum principle. Suppose that we have a point $(p,t)$ so that $\tau(p,t)=0$ and this is the first time when $\tau$ is ever non-positive. Since both $\tau$ and $\partial_s \tau$ are zero, we have the following.
\begin{eqnarray*}
\partial_t \tau(p) & = & \partial_s^2 \tau + 2 (\partial_s \log \kappa) (\partial_s \tau) + 2 \tau \left(\partial^2_s \log \k + \k^3 \right) \\
& = & \partial_s^2 \tau \geq 0,
\end{eqnarray*}
By the strong parabolic maximum principle, $\tau$ must remain strictly positive. \end{proof}
This argument fails at inflection points, where torsion is not defined. As such, this result shows that in order for a Type I singularity to develop from a twisted curve, an inflection point must emerge and afterward a flat point is created where the inflection point occurred.\footnote{It might seem curious that a single flat point will emerge since we initially expect that the torsion must switch signs twice. However, the Frenet-Serret frame bundle can become non-trivial after the emergence of an inflection point.}

\subsection{The main result}

We now turn our attention to proving the main result. To begin, we compute several more time derivatives.
    \begin{eqnarray} \label{Total curvature}
        \partial_t \int_{\gamma_t} \kappa \,ds   &=&  - \int_{\gamma_t} \k\, \tau^2 \, ds 
        \end{eqnarray}
        \begin{eqnarray}
        \label{kappalogkappa}
 \partial_t \int_{\gamma_t} \kappa \log \kappa \,ds   & = & \int_{\gamma_t} (\log \kappa+1) (\partial_s^2 \kappa +\kappa^3-\kappa \tau^2) - \kappa^3 \log \kappa ds \\
  & = & \int_{\gamma_t} -\frac{(\partial_s \kappa)^2}{\kappa} - (\kappa \log \kappa) \tau^2 + \kappa^3 - \kappa \tau^2 \, ds \nonumber \\ \label{kappalogtau}
 \partial_t \int_{\gamma_t} \kappa \log \tau \,ds   & = & \int_{\gamma_t} \left(  \partial_s^2 \kappa + \kappa^3 - \kappa \tau^2 \right)\log \tau \\ 
 & &+ \frac{\kappa}{\tau}\left( 2 \kappa^2 \tau + \partial_s (2 \tau \partial_s \log \kappa) +\partial_s^2 \tau  \right) - \kappa^3 \log \tau \,ds \nonumber \\
%&= & \int_\gamma (-\partial_s \kappa)( \partial_s \log \tau) - \kappa \tau^2 \log \tau + 2 \kappa^3 - 2\frac{(\partial_s \kappa)^2}{\kappa}  \\
%& +&2(\partial_s \kappa)\partial_s \log \tau - (\partial_s \log \tau)( \partial_s \kappa) + \kappa (\partial_s \log \tau)^2 \, ds \nonumber \\
& = & \int_{\gamma_t} - \kappa \tau^2 \log \tau + 2 \kappa^3 - 2\frac{(\partial_s \kappa)^2}{\kappa} + \kappa (\partial_s \log \tau)^2 \, ds \nonumber
\end{eqnarray}

Combining Equations \eqref{kappalogkappa} and \eqref{kappalogtau}, we find the following evolution equation for the curvature-torsion entropy.
\begin{eqnarray} \label{CurvatureTorsionEntropyevolution}
\partial_t \int_{\gamma_t} \kappa \log\left( \frac{\tau}{\kappa^2} \right) \, ds % & = & \partial_t \int_{\gamma_t} \kappa \log \tau \,ds - 2 \partial_t \int_{\gamma_t} \kappa \log \kappa \,ds \\
%& = & \int_{\gamma_t} - \kappa \tau^2 \log \tau + 2 \kappa^3 - 2\frac{(\partial_s \kappa)^2}{\kappa} + \kappa (\partial_s \log \tau)^2 \, ds \nonumber \\ 
 %& & + \int_{\gamma_t} 2 \frac{(\partial_s \kappa)^2}{\kappa} + 2(\kappa \log \kappa) \tau^2 - 2\kappa^3 + 2 \kappa \tau^2 \, ds \nonumber \\
 & = &\int_{\gamma_t} -\kappa \tau^2 \log\left( \frac{\tau}{\kappa^2} \right) + \kappa (\partial_s \log \tau)^2+ 2 \kappa \tau^2 \, ds 
\end{eqnarray}

This immediately provides a second proof that twisted curves cannot have Type I singularities. In particular, for a Type I singularity $\frac{\tau}{\kappa^2}$ goes to zero uniformly (by Altschuler's planarity theorem and the fact that all sequences to the singular time are essential). On the other hand, once $\sup \log\left(\frac{\tau}{\kappa^2}\right)<2$, the curvature-torsion entropy is increasing. 

Furthermore, we can also show that $\log\left(\frac{\tau}{\kappa^2}\right)$ blows up along some sequence. Suppose that there is a uniform bound \[ \log\left(\frac{\tau}{\kappa^2}\right) < C_1. \]
Then Equation \eqref{CurvatureTorsionEntropy} must go to negative infinity as $t$ goes to $\omega$. To see this, consider the region $S \subset \gamma_t$ where $\log\left(\frac{\tau}{\kappa^2}\right)>0$. On this set,
\[\int_S \kappa \log\left(\frac{\tau}{\kappa^2}\right) \,ds \leq C_1 \int_S \kappa \, ds \leq C_1 \int_{\gamma_0} \kappa \, ds, \]
(since the total curvature is decreasing). On the other hand, $\gamma_t$ has a region that converges to a Grim Reaper curve (whose curvature-torsion entropy is $-\infty$). However, we can estimate the left hand side of Equation \eqref{CurvatureTorsionEntropyevolution} by
\begin{eqnarray} 
\partial_t \int_{\gamma_t} \kappa \log\left( \frac{\tau}{\kappa^2} \right) \, ds & \geq& \int_{\gamma_t} - (C_1-2) \kappa \tau^2 \,ds\\
& = & (C_1-2) \partial_t \int_{\gamma_t} \kappa \,ds. \nonumber
\end{eqnarray}
However, since the total curvature decreases to a positive limit as $t$ goes to $\omega$, the difference between the curvature-torsion entropy at the present time and the singular time is bounded, which gives a contradiction.

\subsection{Points with large torsion and curvature}

Theorem \ref{Main Theorem} shows that the potential singularities of a twisted curve are highly unusual, and we suspect they cannot occur. One strategy would be to try to find an essential sequence so that $ \frac{\tau}{\kappa}$ does not go to zero. Although we cannot show this (or even that $\log\left( \frac{\tau}{\kappa^2} \right)$ goes to infinity along a blow-up sequence), we can show a weaker result in this direction.

\begin{proposition}
Suppose $\gamma$ is a curve in $\mathbb{R}^3$ which develops a twisted singularity. Then one of the following two possibilities occurs:
\begin{enumerate}
    \item There exists a sequence $(p_n,t_n)$ so that $\kappa(p_n,t_n)\to 0$.
    \item There is a sequence $(p_n,t_n)$ so that both \[\tau \textrm{ and }\left(2\kappa^2+2\partial_s^2 \log \kappa\right) (\omega-t_n)^\alpha\] tend to infinity for $\alpha<1$. Furthermore, this sequence consists of the points whose torsion is maximized on $\gamma_{t_n}$.
\end{enumerate}
\end{proposition}

 If we assume that $\kappa$ is bounded away from zero near the singular time, then Theorem \ref{Main Theorem} shows that $\log( \tau)$ must go to infinity, and thus we must show that we can extract a sequence so that $(2\kappa^2+2\partial_s^2 \log \kappa)(\omega-t)^\alpha$ blows up.

\begin{eqnarray*}
\partial_t \log (\tau) & = & \frac{1}{\tau} \left(2\k^2\tau + \partial_s \left(\frac{2\tau}{\kappa}
\partial_s \kappa \right) + \partial_s^2 \tau  \right) \\
& = & 2\kappa^2 +2 \partial_s^2 \log \kappa +2(\partial_s \log \tau )(\partial_s \log \kappa)  + \frac{ \partial_s^2 \tau}{\tau} 
\end{eqnarray*}
At the maximum of $\log( \tau)$, the maximum principle implies that
\begin{eqnarray*}
  \partial_s \log \tau  = 0 \hspace{8pt} \textrm{ and } \hspace{8pt}
    \frac{\partial_s^2 \tau}{\tau} \leq 0.
\end{eqnarray*}
At the maximum of $\log(\tau)$, we find that
\[\partial_t \log (\tau) \leq 2\kappa^2 +2 \partial_s^2 \log \kappa. \]
However, in order for $\sup \log (\tau)$ to go to infinity in finite time, Hamilton's maximum principle shows we must be able to extract a subsequence where
\[2\kappa^2+2\partial_s^2 \log \kappa \]
goes to infinity.%\footnote{The reason that we need a more powerful maximum principle is that the estimate on $\partial_t \log (\tau)$ holds only at the point where $\log(\tau)$ is maximized.} 

To show the stronger estimate $(2\kappa^2+2\partial_s^2 \log \kappa(p_n,t_n))(\omega-t)^\alpha \to \infty$, we suppose that this is not the case. Then, we divide up the time interval $(t,\omega)$ into sub-intervals $(t_i, \frac{\omega+t_i}{2})$ and apply the maximum principle on each sub-interval. Doing so, we obtain the estimate
\[ \sup_{t\leq t_n} \log \tau(p,t) -\sup \log \tau (p,t_0)  \leq \sum_{i=0}^{n} \frac{C_2}{2^{i(1-\alpha)}}.  \] This estimate is uniformly bounded in $n$, giving a contradiction. %Here, the process of sub-division is done to ensure that $2\kappa^2+2\partial_s^2 \log \kappa$ remains bounded on each interval, which is needed to apply the maximum principle. 

A similar argument shows that we can extract sequences so that both $((2+\rho)\kappa^2+2\partial_s^2 \log \kappa)(t-\omega)^\alpha$ and $\kappa^\rho \tau$ go to infinity for any $\rho \in (0,1].$ However, we cannot conclude that $\tau$ must go to infinity on such a sequence, so we do not have a geometric application for this fact.

\subsection{A related quantity}

As can be seen from Theorems \ref{Main Theorem} and \ref{Type II singularities}, it is possible to control the behavior of curve shortening flow by finding quantities which are monotone (or nearly monotone) under the flow. There are several other quantities whose evolution is quite simple, and it might be possible to use them to control the singularities which emerge under the flow. For instance, the quantity $\int_{\gamma_t} \tau \log\left( \frac{\tau^2}{\kappa^4} \right) \, ds$ evolves as follows:\footnote{We have squared the terms in the argument of the logarithm so that this quantity is well-defined for arbitrary space curves without assuming they are twisted.}
\begin{eqnarray*}
\partial_t \int_{\gamma_t} \tau \log\left( \frac{\tau^2}{\kappa^4} \right) \, ds  =   \int_{\gamma_t} \kappa^2 \tau \log\left( \frac{\tau^2}{\kappa^4} \right) +  \tau \left((\partial_s \log \kappa^2)^2-\frac{1}{2}(\partial_s \log \tau^2)^2 \right) +4 \tau^3 \, ds.
\end{eqnarray*}

 We do not know of a direct geometric application for this quantity. However, the fact that the evolution is so simple suggests that it may be useful for controlling the formation of singularities. %Note that if a curve flows to a round point, the quantity \[\int_{\gamma_t} \tau \log\left( \frac{\tau^2}{\kappa^4} \right) \, ds \] tends to zero by the planarity theorem.\footnote{Both this expression and the curvature-torsion entropy are \emph{not} scale invariant, so 

\section{A heuristic for the emergence of flat points}

It remains an open question to determine the limiting behavior for generic initial data for spatial curve shortening flow. However, it is reasonable to expect that generic curves converge to round points, the same as for embedded curves in the plane. There are several cases where spatial curve shortening flow is known to converge to a round point. For instance, if the curve is embedded on a standard sphere, we have the following result (a proof can be found in \cite{H}).

\begin{theorem} \label{Spherical curves}
Given a curve $\gamma_0$ embedded on a standard sphere $S^2$ in $\R^3$, under curve shortening flow $\gamma_t$ remains embedded on a shrinking sphere and converges to a round point.
\end{theorem}

Note that the total torsion of a spherical curve is zero, which immediately implies that any such curve has flat points. In fact, Sedykh's theorem shows there are at least four such points \cite{S}. Furthermore, in recent work, Litzinger showed that curves whose entropy is at most $2$ converge to round points \cite{Litzinger}.

\begin{theorem}
     Suppose that ${\gamma}$ is a smooth curve whose entropy 
     \begin{equation}
\lambda({\gamma})=\sup _{x_0 \in \mathbb{R}^n, t_0>0}\left(4 \pi t_0\right)^{-\frac{1}{2}} \int_{\gamma_t} \mathrm{e}^{-\frac{\left|x-x_0\right|^2}{4 t_0}} d \mu  
\end{equation}
satisfies $\lambda({\gamma}) \leq 2$
Then under curve shortening flow, ${\gamma_t}$ converges to a round point.
\end{theorem}
 
Apart from these results, there is another reason why we do not expect initially twisted curves to remain twisted; the reaction terms for the curvature and torsion tend to create points of vanishing curvature. To see this, we ignore the spatial derivative terms in Equation \eqref{eq:Curvature Evolution} describing the evolution of $\kappa$ and $\tau$ and simply consider the system of coupled ODEs
 \begin{equation}
     \begin{aligned}
         \dot \kappa &=& \kappa^3-\kappa \tau^2 \\
         \dot \tau & = & 2 \kappa^2 \tau.
     \end{aligned}
 \end{equation}
This system reduces to the homogeneous equation
 \begin{equation*}
\frac{d \kappa}{d \tau} = \frac{1}{2} \left( \frac{\kappa}{\tau}-\frac{\tau}{\kappa} \right),
 \end{equation*}
whose solutions are circular arcs of the form
\[\kappa(\tau) = \sqrt{ C \tau - \tau^2}.\]
As such, no matter how large $\kappa$ is initially, the reaction ODE tends towards a situation where $\kappa$ vanishes. The full evolution equation for the torsion also includes a complicated term involving its derivatives and those of the curvature, so this analysis does not constitute a complete proof. Note however, that if the initial curve is a helix, which is a curve of constant curvature and torsion, this calculation shows that the limiting configuration is a straight line, but the limiting values of the torsion is non-zero.\footnote{The limiting value of the torsion is not the torsion of a straight line, which is undefined. Indeed, the limiting value depends on the initial curvature and torsion of the helix.}

\subsection*{Acknowledgements}
The author would like to thank Xuan Hien Nguyen, Mizan Khan and Kori Khan for their helpful comments. He is partially supported by Simons Collaboration Grant 849022 (``K\"ahler-Ricci flow and optimal transport").

\begin{comment}
\section{Bibliography}
\end{comment}


\begin{thebibliography}{99}
\bibitem{AL} U.~Abresch and J.~Langer, The Normalized Curve Shortening Flow and Homothetic Solutions, \emph{J. Differential Geom},\textbf{23} (1986),  175--196

\bibitem{A} S.~J.~Altschuler, Singularities of the Curve Shrinking Flow for Space Curves, \emph{J. Differential Geom} 
\textbf{34} (1991), 491--514.

\bibitem{AG} S.~J.~Altschuler and M.~A.~Grayson, Shortening Space Curves and Flow Through Singularities, \emph{J. Differential Geom} 
\textbf{35} (1992), 283--298.

\bibitem{Costa} S.~I.~Rodrigues Costa, "On closed twisted curves." Proceedings of the American Mathematical Society 109, no. 1 (1990): 205--214.

\bibitem{GH} M.~Gage and R.~S.~Hamilton, The Heat Equation Shrinking Convex Plane Curves,\emph{J. Differential Geom} \textbf{23} (1986), 69--96.

\bibitem{G} M.~A.~Grayson, The Heat Equation Shrinks Embedded Plane Curves To Round Points,\emph{J. Differential Geom} \textbf{26} (1987), 285--314.

\bibitem{Khan} G. Khan,  A condition ensuring spatial curves develop type-II singularities under curve shortening flow. (2012), arXiv preprint arXiv:1209.4072.

\bibitem{Litzinger} F. Litzinger, Singularities of low entropy high codimension curve shortening flow. (2023), arXiv preprint arXiv:2304.02487.

\bibitem{S} V.~D.~Sedykh, Four vertices of a convex space curve, Bull. London Math. Soc. 26 (1994), 177-180.

\bibitem{H} H.~Siming, Distance Comparison Principle and Grayson Type Theorem in the Three Dimensional Curve Shortening Flow

\bibitem{YanJiao} Y.~Y. ~Yang and X.~X.~ Jiao. Curve shortening flow in arbitrary dimensional Euclidian space, \emph{Acta Mathematica Sinica} \textbf{21} (2005): 715--722.





\end{thebibliography}
\end{document}